\documentclass[11pt]{elsarticle} 
\usepackage{amsmath}
\usepackage{amssymb} 
\usepackage{pgf}
\usepackage{tikz}
\usetikzlibrary{matrix}
\usetikzlibrary{arrows}
\tikzstyle{place}=[draw,circle,minimum size=1.2mm,inner sep=1pt,outer sep=-1.1pt,fill=black]
\tikzset{>=latex,shorten >= 0.01cm,shorten <= 0.01cm}
\usepackage{array, multirow}
\DeclareMathOperator{\sgn}{sgn}

\newcommand{\A}{\mathcal{A}}
\newcommand{\R}{\mathbb{R}}

\newcommand{\cA}{\mathcal{A}}\newcommand{\cT}{\mathcal{T}}
\newcommand{\cB}{\mathcal{B}}

\newcommand{\K}{\mathcal{K}}
\newcommand{\W}{\mathcal{W}}
\newcommand{\V}{\mathcal{V}}\newcommand{\U}{\mathcal{U}}
\newcommand{\T}{\mathcal{T}}
\newcommand{\B}{\mathcal{B}}
\newcommand{\G}{\mathcal{G}}
 
\newcommand{\ri}[1]{\mathop{\rm ri}(#1)} 

\newcommand{\kk}{\ell}
\newcommand{\kl}{\ell}
\newcommand{\elb}{k}

\newtheorem{theorem}{Theorem}[section] 
 
\newtheorem{corollary}[theorem]{Corollary}

\newtheorem{example}[theorem]{Example}

\newtheorem{remark}[theorem]{Remark}

\begin{document}

\title{Bordering for spectrally arbitrary sign patterns\footnote{\copyright\ 2017. This manuscript version is made available under the CC-BY-NC-ND 4.0 
		license http://creativecommons.org/licenses/by-nc-nd/4.0/ }}

\date{August 3, 2017} 

\author[O]{D.D. Olesky} 
\ead{dolesky@cs.uvic.ca}
\author[vdD]{P. van den Driessche}
\ead{pvdd@math.uvic.ca}
\author[VM]{K.~N. Vander~Meulen\corref{cor1}} 
\ead{kvanderm@redeemer.ca}

\cortext[cor1]{Corresponding author}

\address[O]{Department of Computer Science, University of Victoria, BC, Canada, V8W 2Y2}
\address[vdD]{Department of Mathematics and Statistics, University of Victoria, BC, Canada, V8W 2Y2} 
\address[VM]{Department of Mathematics, Redeemer University College, ON, Canada, L9K 1J4}

\begin{abstract}
We develop a matrix bordering technique that can be applied to an irreducible spectrally arbitrary 
sign pattern to construct a higher order spectrally arbitrary sign pattern. This technique
generalizes a recently developed triangle extension method. We describe recursive constructions
of spectrally arbitrary patterns using our bordering technique, and show
that a slight variation of this technique  can be used to construct inertially
arbitrary sign patterns. 
\end{abstract}

\begin{keyword}{nilpotent matrix, spectrally arbitrary pattern, nilpotent-Jacobian method, inertially arbitrary pattern.}
\\

\emph{2010 Mathematics Subject Classification}. 15A18, 15B35.
\end{keyword}
\maketitle 


\tikzstyle{place}=[circle,draw=black!100,fill=black!100,thick,inner sep=0pt,minimum size=1mm]
\tikzstyle{left}=[>=latex,<-,semithick]
\tikzstyle{right}=[>=latex,->,semithick]
\tikzstyle{nleft}=[>=latex,-,semithick]
\tikzstyle{nright}=[>=latex,-,semithick]
\tikzstyle{right2}=[-,semithick]

\section{Introduction}
A number of methods have been developed to check that a specific pattern
is spectrally or inertially arbitrary, such as the analytic nilpotent-Jacobian method and 
the algebraic nilpotent-centralizer method (see e.g. \cite{CGKOVV, DJOD, GS, GS2}), 
and these have been applied to various classes of 
patterns (see e.g. \cite{Britz, CV, P}). Recently in~\cite{KSVW}, a digraph method called {\emph{triangle extension}}
has been developed for constructing higher order spectrally or inertially arbitrary patterns
from lower order patterns. In this paper, we generalize the triangle extension method
by formulating it as a matrix bordering technique (see Remark~\ref{te}).   
With this bordering technique, we construct higher order patterns (some of which cannot
be obtained by triangle extension) that are spectrally or inertially arbitrary from lower order patterns.
We give examples of new spectrally and inertially arbitrary sign patterns obtained by bordering.

\subsection{Definitions and the nilpotent Jacobian method.}
Given an order $n$  matrix $A=[a_{ij}]$, denote the characteristic polynomial of 
$A$ by $p_A(z)=\det(zI-A)$.  
A \emph{sign pattern} is a matrix $\A=[\alpha_{ij}]$ of order $n$ with entries in $\{0,+,-\}$.
Let $$Q(\A)=\{ A\ | \  a_{ij}=0 {\rm{\ if\ }} \alpha_{ij}=0,  a_{ij}>0 
{\rm{\ if\ }} \alpha_{ij}=+  {\rm{\ and \ }} a_{ij}<0 
{\rm{\ if\ }} \alpha_{ij}=-\}.$$
If $A\in Q(\A)$ for some pattern $\A$, then  $A$ is a  \emph{realization} of $\A$
and we sometimes refer to $\A$ as $\sgn(A)$.
A pattern $\A$ is \emph{spectrally arbitrary} if for every degree $n$ monic polynomial
$p(z)$ over $\R$, there is some real matrix $A$ such that $A\in Q(\A)$ and $p_A(z)=p(z)$.
A pattern $\B=[\beta_{ij}]$ is a \emph{superpattern} 
of $\A$ if $\alpha_{ij}\neq 0$ implies $\beta_{ij}=\alpha_{ij}$, and $\A$ is a \emph{subpattern} of $\B$.
Two patterns $\A$ and $\B$ are {\emph{equivalent}} if $\B$ can be obtained
from $\A$ via any combination of negation, transposition, permutation similarity and signature similarity. 

A matrix $A$ is \emph{nilpotent} if $A^k=0$ for some positive integer $k$ and  
the smallest positive integer $k$ such that $A^k=0$ is the \emph{index} of $A$.
An order $n$ nilpotent matrix $A$ has characteristic polynomial $p_A(z)=z^n$.

Suppose $\A$ is an order $n$ sign pattern with a nilpotent matrix $A\in Q(\A)$ with
$m\geq n$ nonzero entries $a_{i_1j_1}, a_{i_2j_2},\ldots, a_{i_mj_m}$. Let 
$X=X_A(x_1,x_2,\ldots,x_m)$ denote the matrix obtained from $A$ by replacing
$a_{i_kj_k}$ with the variable $x_k$ for $k=1,\ldots, m$. Writing
$p_{X}(z)=z^n+f_1z^{n-1}+\cdots+f_{n-1}z+f_n$
for some $f_i=f_i(x_1,x_2,\ldots,x_m)$, let $J=J_{X}$ be the 
$n \times m$ Jacobian matrix with $(i,j)$ entry equal to $\frac{\partial f_i}{\partial x_j}$ for
$1\leq i \leq n$, and $1\leq j \leq m$. Let $J_{X=A}$ denote the Jacobian
matrix evaluated at the nilpotent realization, that is 
$J_{X=A}=J\vert_{(x_1,x_2,\ldots,x_m)=(a_{i_1j_1}, a_{i_2j_2},\ldots, a_{i_mj_m})}$. 
A nilpotent matrix $A$ \emph{allows
a full-rank Jacobian} if the rank of $J_{X=A}$ is $n$.
Finding a nilpotent matrix $A\in Q(\A)$ that allows a full-rank 
Jacobian is known as the \emph{nilpotent-Jacobian method}. As noted in part (c) of Theorem~\ref{SAP},
 this method guarantees
that every superpattern of $\A$ is spectrally arbitrary.

A matrix $A$ (or pattern $\A$) is \emph{reducible} 
 if there is a permutation matrix $P$ such that $PAP^T$ (resp. $P\A P^T$)
 is block triangular with more than one nonempty diagonal block. Otherwise it is \emph{irreducible}. 
A matrix $A$ is  
\emph{nonderogatory} if the dimension of the eigenspace of every 
eigenvalue is equal to one.
The following theorem combines known results from 
\cite{CGKOVV}
and \cite{DJOD}.

\begin{theorem}\label{SAP}
Let $\A$ be a sign pattern of order $n$. If a nilpotent matrix $A\in Q(\cA)$ allows a full-rank Jacobian, then 
\begin{enumerate}
\item[(a)] $A$ is irreducible,
\item[(b)] $A$ is nonderogatory, and
\item[(c)] every superpattern of $\cA$ is spectrally arbitrary. 
\end{enumerate}
\end{theorem}

\begin{proof}
Suppose $A\in Q(\cA)$ is a nilpotent matrix of order $n$ that
 allows a full-rank Jacobian. 
Part (c) is \cite[Theorem 3.1]{CGKOVV}, which is a
reframing of the nilpotent-Jacobian method introduced in 
\cite{DJOD}. Part (b) is \cite[Corollary 4.5]{CGKOVV}.  

If $A$ is a reducible nilpotent  matrix and $PAP^T$
is block triangular for some permutation matrix $P$, 
then the index of $A$ is at most the index of the largest order diagonal block
of $PAP^T$. Thus the index is bounded above by the order of the largest diagonal
block. Since the index of $A$ is $n$, it follows that $A$ is irreducible, proving part (a).
\end{proof}

Because of part (c) of Theorem~\ref{SAP}, \emph{minimal} spectrally arbitrary
patterns (that is, spectrally arbitrary patterns for which no proper subpattern
is spectrally arbitrary), are of special interest. For $n=2$ and $n=3$, the minimal spectrally
arbitrary patterns are well-known (see, e.g., \cite{Britz,CV}) and, up to equivalence,
are: 
$$\T_2=\left[\begin{array}{cc}
+&-\\
+&-
\end{array}
\right], \qquad \T_3=\left[\begin{array}{ccc}
+&-&0\\+&0&-\\0&+&-\end{array}\right],$$
$$ 
\U_3=\left[\begin{array}{ccc}
+&-&+\\+&-&0\\+&0&-\end{array}\right],
\ \V_3=\left[\begin{array}{ccc}
+&-&0\\+&0&-\\+&0&-\end{array}\right], {\hbox{\rm{\ and\ }}}
\ \W_3=\left[\begin{array}{ccc}
+&+&-\\+&0&-\\+&0&-\end{array}\right].
$$

\subsection{Bordering}

Let $A=[a_{ij}]$ be an order $n$ matrix, $\mathbf{x},\mathbf{z}\in\R^n$, 
and let $B$ be the \emph{bordered} matrix of order $n+1$: 
\begin{equation}\label{B}
B=
\left[\begin{array}{cc}
I_n&\mathbf{0}\\
\mathbf{x}^T&1
\end{array}\right]
\left[\begin{array}{cc}
A&\mathbf{z}\\
\mathbf{0}^T&0
\end{array}\right]
\left[\begin{array}{cc}
I_n&\mathbf{0}\\
-\mathbf{x}^T&1
\end{array}\right]=
\left[\begin{array}{c|c}
A-\mathbf{z}\mathbf{x}^T&\mathbf{z}\\ \hline
\mathbf{x}^T(A-\mathbf{z}\mathbf{x}^T)&\mathbf{x}^T\mathbf{z}
\end{array}\right].
\end{equation}
Since this is a similarity transformation,
it follows that $p_B(z)=zp_A(z)$, and thus $B$ is nilpotent if $A$ is nilpotent.
Note that (\ref{B}) is a special case of a construction introduced
in \cite[Theorem 3.1]{KOSVVV}.

Let $\mathbf{e}_i =[0,\dots,0,1,0,\ldots,0]^T$ with a $1$ in position $i$.
In this paper, we focus on the special cases 
$\mathbf{z}=\mathbf{e}_j$ 
and $\mathbf{x}^T=b \mathbf{e}_k$  
for some $b\neq 0$, which we call {\em{standard unit bordering}}.
In the next two sections we use bordering to construct higher order   
spectrally arbitrary patterns out of lower order patterns without having to recalculate 
a Jacobian matrix. In addition, at each stage, the construction provides
an explicit nilpotent realization of the spectrally arbitrary pattern.

\section{Standard unit bordering with equal indices}\label{unitequal}

Let $A=[a_{ij}]$, and denote the $k$th row of $A$ by $r_k(A)$.
Suppose $\mathbf{x}=a_{kk}\mathbf{z}=a_{kk}\mathbf{e}_k$ for some $a_{kk}\neq0$. 
Then $\mathbf{x}^TA=a_{kk}r_k(A)$ and $A-\mathbf{z}\mathbf{x}^T=A-a_{kk}P_{kk}$ where $P_{kk}$ has a $1$ in entry
$(k,k)$ and zeros elsewhere. In this case, the matrix $B$ in  (\ref{B}) is  
\begin{equation}\label{Be} 
B=\left[\begin{array}{c|c}
A-a_{kk}P_{kk}&\mathbf{e}_k\\ \hline
a_{kk}r_k(A-a_{kk}P_{kk})&a_{kk} \end{array}\right].
\end{equation}
 Let $A(u,v)$ denote the matrix obtained from $A$ by deleting
row $u$ and column $v$. 

\begin{theorem}\label{borderT}
Let $\A$ be a sign pattern of order $n$. Suppose $A=[a_{ij}]\in Q(\A)$ is a nilpotent matrix
and $A$ allows a full-rank Jacobian. 
Suppose $a_{kk}\neq0$ and $a_{kv}\neq 0$ for some $v\neq k$.
If $\det{A(k,v)}\neq 0$, 
then $B$  in $(\ref{Be})$ is a nilpotent matrix that allows a full-rank Jacobian and
hence every superpattern of $\cB=\sgn(B)$ is spectrally arbitrary. 
\end{theorem}

\begin{proof} Let $A\in Q(\A)$ be a nilpotent matrix and $X_A$ be a matrix with the nonzero pattern of
$\A$ having variable entries
such that the Jacobian $J_{X_A=A}$
has rank $n$. For convenience, assume $k=n$ and 
the last row of $X_A$ is $[x_{n1},x_{n2},\ldots, x_{nn}]$, recognizing
that some of these entries may be zero. Note that by assumption $x_{nv}$ and $x_{nn}$ are nonzero. 
Let $B$ be as in (\ref{Be}) and 
\begin{equation}\label{XB}
X_B=
\left[\begin{array}{ccc|c}
&&&\\
&X_A-x_{nn}P_{nn}&&\mathbf{0}\\ 
&&&1\\ \hline
&\mathbf{y}&0&x_{nn}
\end{array}
\right]
\end{equation}
with $\mathbf{y}=[y_1,y_2,\ldots, y_{n-1}]$ such that $y_i\neq 0$
if and only if  $x_{ni}\neq 0$. (Note that, other than the placement of the
variables in $\mathbf{y}$, the nonzero entries of $\mathbf{y}$ are 
independent of the variables in $X_A$.) Then $X_B$ has the nonzero pattern of $\cB$.
Using cofactor expansion along the last row of $X_B$ gives  
\begin{eqnarray*}
p_{X_B}(z)&=&\det(zI_{n+1}-X_B)\\
      &=&(z-x_{nn})\det(zI_n-X_A+x_{nn}P_{nn})+\sum_{\kl =1}^{n-1}(-1)^{n+\kl }y_\kl \det\left([zI_n-X_A](n,\kl )\right).\\      
\end{eqnarray*}
However, applying cofactor expansion along the last row of the first summand gives
\begin{eqnarray*}
\det(zI_n-X_A+x_{nn}P_{nn})&=&z\det\left([zI_n-X_A](n,n)\right)\\
         & & +\sum_{\kl =1}^{n-1}(-1)^{n+\kl }x_{n\kl }\det\left([zI_n-X_A](n,\kl )\right).
\end{eqnarray*}
Thus
\begin{eqnarray*}
p_{X_B}(z)&=&z\det(zI_n-{X_A}+x_{nn}P_{nn})-x_{nn}z\det([zI_n-X_A](n,n))\\
&+& \sum_{\kl =1}^{n-1}(-1)^{n+\kl }(y_\kl -x_{nn} x_{n\kl })\det\left([zI_n-X_A](n,\kl )\right).
\end{eqnarray*}
Since the determinant is linear in the rows (or using a rank 1 perturbation of a determinant), it follows that
\begin{equation}\label{BB}
p_{X_B}(z)=zp_{X_A}(z)+ \sum_{\kl =1}^{n-1}(-1)^{n+\kl }(y_\kl -x_{nn} x_{n\kl })\det\left([zI_n-X_A](n,\kl )\right).
\end{equation}
Focusing on the coefficients of $p_{X_{B}}(z)$, the second summand can be rewritten as
$$\sum_{r=3}^{n+1}\left[\sum_{\kl =1}^{n-1}S_{r,\kl }(y_\kl -x_{nn}x_{n\kl })\right]z^{n-r+1}$$
for some polynomials $S_{r,\kl }$ of the variable entries in $X_A$. 
To consider the Jacobian of $X_B$, we assume the last columns of $J_{X_B}$ are
indexed by the nonzeros of $x_{n1},\ldots,x_{nn},y_1,\ldots,y_{n-1}$. 
Let $m$ be the number of nonzero entries of $y$ and $w$ be the number of variables
in $X_A$. Then the $(n+1)\times(w+m)$ Jacobian matrix
$J_{X_B}$ is 
\begin{equation}\label{JacB}
J_{X_B}=
\left[\begin{array}{cc}
J_{X_A} & O\\
\mathbf{0}^T&\mathbf{0}^T
\end{array}\right]
+\sum_{\kl =1}^{n-1}(y_\kl -x_{nn}x_{n\kl })M_\kl +
\left[\begin{array}{cc}
O & N\\
\end{array}\right]
\end{equation}
for some matrices $M_\kl $ and $(n+1) \times (2m+1)$ matrix $N$ with columns indexed by 
the nonzeros of $x_{n1},\ldots,x_{nn},y_1,\ldots,y_{n-1}$. 
Note that by (\ref{B}) and (\ref{XB}), $y_\kl = a_{nn}a_{n\kl }=x_{nn}x_{n\kl }$ in the nilpotent
realization, so that we can ignore each matrix $M_\kl $ in (\ref{JacB}), since its
 coefficient vanishes at the nilpotent realization.
Further the column of $N$
corresponding to $y_\kl $ is
$\overrightarrow{N}_{y_\kl }=[0,0,S_{3,\kl },S_{4,\kl },\ldots,S_{n+1,\kl }]^T$ for $1\leq \kl \leq n$,
and in addition, the column corresponding to $x_{n\kl }$ is  
$\overrightarrow{N}_{x_{n\kl }}=-x_{nn}\overrightarrow{N}_{y_\kl }$ for $1\leq \kl \leq n-1$ 
and $\overrightarrow{N}_{x_{nn}}=\sum_{\kl =1}^{n-1}-x_{n\kl }\overrightarrow{N}_{y_\kl }$.
It follows that $N$ is column equivalent 
to $[\ O\ |\overrightarrow{N}_{y_1}|\overrightarrow{N}_{y_2}|\cdots|\overrightarrow{N}_{y_{n-1}}].$
From (\ref{BB}), with $z=0$, 
$$S_{n+1,\kl }=(-1)^{\kl -1}\det(X_A(n,\kl )),$$
giving
$$\left.S_{n+1,\kl }\right\vert_{X_B=B}=(-1)^{\kl -1}\det(A(n,\kl )).$$
Thus, the condition that 
there exists an index $v\neq n$ such that $a_{nv}\neq 0$ and $\det{A(n,v)}\neq 0$ implies
that $\left.S_{n+1,\kl }\right\vert_{X_B=B}\neq 0$ for some $\kl$, $1\leq \kl \leq n-1$. It follows that
$J_{X_B=B}$ is equivalent to
$$
\left[\begin{array}{cc}
J_{X_A=A} & *\\
\mathbf{0}^T&\mathbf{s}^T
\end{array}\right]
$$
for some $\mathbf{s}\neq 0$. Hence $B$ allows a full-rank Jacobian.
Thus by Theorem~\ref{SAP}, every superpattern of $\cB$ is spectrally 
arbitrary. 
\end{proof}

\begin{remark}\label{te}{\rm 
A method in \cite{KSVW} called \emph{triangle extension on arc} $(u,v)$ (in the digraph
associated with $\A$) 
is equivalent to a special case of 
applying Theorem~\ref{borderT} to row $u$ of $A$ and 
entry $(u,v)$, 
namely in the situation that $a_{uu}$ and $a_{uv}$ are the only nonzero entries in row $u$ of $A$.
}
\end{remark}

\begin{example}\label{E1}
{\rm If  
\[\A=\left[\begin{array}{cccc} 
0&+&0&0\\
0&-&+&0\\
+&0&0&+\\
+&0&-&+\\
\end{array}
\right]
\mbox{\rm{\qquad and \qquad}}
A=\left[\begin{array}{rrrr} 
0&1&0&0\\
0&-1&1&0\\
1&0&0&1\\
1&0&-1&1\\
\end{array}
\right],
\]
then $A$ is nilpotent, $A\in Q(\A)$, and $A$ allows a full-rank Jacobian. Hence $\cA$ is spectrally arbitrary ($A$ is equivalent to
the second matrix in Appendix A of \cite{CM}).
Further, $a_{44}\neq 0$, $a_{41}\neq 0$ and  $\det(A(4,1))\neq 0$. Applying Theorem~\ref{borderT}   
to row $4$ and entry $(4,1)$ gives a
spectrally arbitrary pattern $\cB_5$ with nilpotent matrix $B\in Q(\cB_5)$ for
\[ \cB_5=\left[\begin{array}{ccccc} 
0&+&0&0&0\\
0&-&+&0&0\\
+&0&0&+&0\\+&0&-&0&+\\
+&0&-&0&+\\
\end{array}
\right]
{\mbox{\rm{\qquad and \qquad}}}
B=\left[\begin{array}{rrrrr} 
0&1&0&0&0\\
0&-1&1&0&0\\
1&0&0&1&0\\1&0&-1&0&1\\
1&0&-1&0&1\\
\end{array}
\right].
\]
Note that since row $4$ of $A$ has more than one off-diagonal entry, triangle extension as described \cite{KSVW}
is not possible on the arc $(4,1)$ in the digraph associated with $\A$, demonstrating that Theorem~\ref{borderT} provides a more general technique than triangle extension in \cite{KSVW}.
}
\end{example}

\begin{remark} {\rm
Theorem~\ref{borderT} can be applied recursively. In particular, 
suppose $\det(A(n,v))\neq 0$ and Theorem~\ref{borderT} was applied to row $n$ and entry $(n,v)$ of $A$ to obtain $B$. 
It follows that  
$\det(B(n+1,v))=(-1)^n\det(A(n,v))\neq 0$ since there is only one nonzero
entry in the last column of $B(n+1,v),$ namely $1$ in the last row.  
Thus Theorem~\ref{borderT} can now be applied to row $n+1$ and entry $(n+1,v)$ of $B$. 
}
\end{remark}

\begin{example}\label{BN} {\rm 
The sign pattern $\B_5$ in Example~\ref{E1} can be recursively bordered using Theorem~\ref{borderT}, starting with 
row $5$ and entry $(5,1)$, to obtain
a spectrally arbitrary pattern of order $n\geq 6$, with $3n-4$ nonzero entries, of the form   
\[\B_n=\left[\begin{array}{ccccccc} 
0&+&&&&&\\
0&-&+&&&&\\
+&0&0&+&&O&\\
+&0&-&0&+&&\\
\vdots&\vdots&\vdots&\vdots&\ddots&\ddots&\\
+&0&-&0&\cdots&0&+\\
+&0&-&0&\cdots&0&+
\end{array}
\right].\]
Note that each nonzero entry of the nilpotent realization of $\cB_n$ has magnitude $1$.
It can be shown that $\B_5$ and $\B_4=\A$ in Example \ref{E1} are minimally spectrally arbitrary. 
}
\end{example}

\section{Standard unit bordering with unequal indices}\label{unitunequal}

Referring to (\ref{B}), suppose $\mathbf{x}=b \mathbf{e}_\elb$ for some $b\neq 0$ and 
$\mathbf{z}=\mathbf{e}_j$ for $j\neq \elb$; thus
$\mathbf{x}^T\mathbf{z}=0$.  
With the $\elb$th row of $A$ denoted by $r_\elb (A)$, 
 $\mathbf{x}^TA=br_\elb (A)$ and $A-\mathbf{z}\mathbf{x}^T=A-bP_{j\elb}$ where $P_{j\elb}$ has a $1$ in entry
$(j,\elb)$ and zeros elsewhere. In this case, the matrix $B$ in  (\ref{B}) is 
\begin{equation}\label{B2}
B=\left[\begin{array}{c|c}
A-bP_{j\elb}&\mathbf{e}_j\\ \hline
br_\elb (A-bP_{j\elb})&0 \end{array}\right].
\end{equation}

Recall that $X_A$ is obtained from $A$ be replacing some of the nonzero entries with variables. In the case that
$J_{X=A}$ has rank $n$, we call a nonzero entry of $A$ \emph{Jacobian in $X_A$} if it is replaced by a
variable in $X_A$, otherwise the entry is \emph{non-Jacobian in $X_A$}. 
Note that a non-Jacobian entry may be zero. 
To simplify the next proof, for $U,V\subseteq \{1,2,\ldots,n\}$,  
let $A(U,V)$ denote the matrix obtained from $A$ by deleting
the rows in $U$ and the columns in $V$. 

\begin{theorem}\label{borderT2}
Let $\A$ be a sign pattern of order $n$. Suppose $A=[a_{ij}]\in Q(\A)$ is a nilpotent matrix 
and $A$ allows a full-rank Jacobian. Suppose $a_{j\elb }$, $j\neq \elb$, is non-Jacobian for some 
choice of $X_A$. 
If $a_{\elb v}\neq0$ 
and $\det{A(j,v)}\neq 0$, for some $v$, 
then $B$  in $(\ref{B2})$ is a nilpotent matrix that allows a full-rank Jacobian and
hence every superpattern of $\cB=\sgn(B)$ is spectrally arbitrary. 
\end{theorem}

\begin{proof}
Let $A\in Q(\A)$ be a nilpotent matrix and $X_A$ be a matrix with the nonzero pattern of
$\A$ having variable entries
such that the Jacobian $J_{X_A=A}$
has rank $n$ with no variable placed in position $(j,\elb)$.  
For convenience, assume that $j=1$, $\elb=n$, and the last row of $X_A$ is $[x_{n1},x_{n2},\ldots, x_{nn}]$, recognizing
that some of these entries may be zero. 
Let $B$ be as in (\ref{B2}) and
\begin{equation}\label{XBQ}
X_B=
\left[\begin{array}{ccc|c}
&&&1\\
&X_A-bP_{1n}&&\mathbf{0}\\ 
&&&\\ \hline
&\mathbf{y}&&0
\end{array}
\right]
\end{equation}
with $\mathbf{y}=[y_1,y_2,\ldots, y_{n}]$ such that $y_i\neq 0$
if and only if  $x_{ni}\neq 0$. (Note that, other than the placement of the
variables in $\mathbf{y}$, the nonzero entries of $\mathbf{y}$ are 
independent of the variables in $X_A$.) Then $X_B$ has the nonzero pattern of $\cB$.
Using cofactor expansion along the last row of $X_B$ gives
\begin{eqnarray}
p_{X_B}(z)&=&\det(zI_{n+1}-X_B)\nonumber \\ 
      &=&z\det(zI_n-X_A+bP_{1n})+\sum_{\kk=1}^{n}(-1)^{\kk}y_\kk\det\left([zI_n-X_A](1,\kk)\right) \nonumber\\      
	 &=& zp_{X_A}(z) +(-1)^{n+1}zb\det\left([zI_n-X_A](1,n)\right) \label{H1}\\
	 & & \qquad\qquad  +\sum_{\kk=1}^{n}(-1)^{\kk}y_\kk\det\left([zI_n-X_A](1,\kk)\right).\nonumber
\end{eqnarray}
Let $W_\kk=z\det\left([zI_n-X_A](\{1,n\},\{\kk,n\})\right)$. Applying cofactor 
expansion on the determinant in the second summand of (\ref{H1}) gives
\begin{eqnarray}\label{H2}
z\det\left([zI-X_A](1,n)\right)=\sum_{\kk=1}^{n-1}(-1)^{n+\kk}x_{n\kk}W_\kk.
\end{eqnarray}
Using the fact that the last row of $zI_n-X_A$ is $[0,\cdots,0,z]-[x_{n1},x_{n2},\ldots, x_{nn}],$ 
and that a determinant is linear in the last row, 
\begin{eqnarray}\label{H3}
\sum_{\kk=1}^{n}(-1)^{\kk}y_\kk\det\left([zI_n-X_A](1,\kk)\right)
=\sum_{\kk=1}^{n-1}(-1)^{\kk}y_\kk W_\kk 
+\sum_{\kk=1}^{n}(-1)^{\kk}y_\kk U_\kk
\end{eqnarray}
for $$U_\kk=\det\left([zI_n-X_A](1,\kk)-z\left[\begin{array}{cc}O&\mathbf{0}\\\mathbf{0}^T&1\end{array}\right]\right),
\qquad {\rm{with \quad}} 1\leq \kk \leq n-1,$$
and $U_n=\det([zI_n-X_A](1,n))$.
Using (\ref{H2}) and (\ref{H3}) in (\ref{H1}) gives
\begin{eqnarray*}
p_{X_B}(z)&=&zp_{X_A}(z)+\sum_{\kk=1}^{n-1}\left((-1)^{\kk+1}bx_{n\kk}W_\kk+(-1)^{\kk}y_\kk W_\kk  \right)
+\sum_{\kk =1}^{n}(-1)^{\kk }y_\kk  U_\kk .
\end{eqnarray*}
However, using cofactor expansion along the last row of the matrix in $U_\kk $ gives
\begin{eqnarray*}
U_\kk  &=& \sum_{i=1}^{\kk -1}(-1)^{n+i}x_{ni}\det\left([zI_n-x_A](\{1,n\},\{i,\kk \}) \right)\\
&&+\sum_{i=\kk +1}^n(-1)^{n+i-1}x_{ni}\det\left([zI_n-x_A](\{1,n\},\{\kk ,i\}) \right).
\end{eqnarray*}
Thus 
\begin{eqnarray*}
p_{X_B}=zp_{X_A}&+&\sum_{\kk =1}^{n-1}(-1)^{\kk }(y_\kk -bx_{n\kk })W_\kk  \nonumber\\
&+&\sum_{1\leq i<\kk \leq n}(y_\kk x_{ni}-y_ix_{n\kk })(-1)^{n+i+\kk }\det\left([zI_n-X_A](\{1,n\},\{\kk ,i\})\right).
\end{eqnarray*}
Focusing on the coefficients of $p_{X_{B}}(z)$, we can rewrite $p_{X_B}(z)$ as
\begin{equation}\label{pB}
zp_{X_A}(z)+\sum_{r=3}^{n+1}\left[\sum_{\kk =1}^{n-1}S_{r,\kk }(y_\kk -bx_{n\kk })\right]z^{n-r+1} +
\sum_{r=5}^{n+1}\left[\sum_{1\leq i<\kk \leq n}T_{r,i,\kk }(y_\kk x_{ni}-y_ix_{n\kk })\right]z^{n-r+1}
\end{equation}
for some polynomials $S_{r,\kk }$ and $T_{r,i,\kk }$ in the variable entries of $X_A$. 
Note that the variables in the last row of $X_A$ do not appear in $S_{r,\kk }$ or $T_{r,i,\kk }$.
To consider the Jacobian of $X_B$, we assume the last columns of $J_{X_B}$ are
indexed by the nonzeros of $x_{n1},\ldots,x_{nn}$, and $y_1,\ldots,y_{n}$. 
Let $m$ be the number of nonzero entries of $y$ and $w$ be the number of variables
in $X_A$. Since $x_{1n}$ is non-Jacobian in $X_A$, the $(n+1)\times(w+m)$ Jacobian matrix
$J_{X_B}$ is 
\begin{equation}\label{JacB2}
J_{X_B}=
\left[\begin{array}{cc}
J_{X_A} & O\\
\mathbf{0}^T&\mathbf{0}^T
\end{array}\right]
+\sum_{\kk =1}^{n-1}(y_\kk -bx_{n\kk })M_\kk 
+\sum_{1\leq i < \kk \leq n}(y_\kk x_{ni}-y_ix_{n\kk })H_{i,\kk }
+
\left[\begin{array}{cc}
O & N\\
\end{array}\right]
\end{equation}
for some matrices $M_\kk $, $H_{i,\kk }$ and $(n+1) \times (2m)$ matrix $N$ with columns indexed by 
the nonzeros of $x_{n1},\ldots,x_{nn},y_1,\ldots,y_{n}$.
Note that by (\ref{B2}) and (\ref{XBQ}), $y_\kk = ba_{n\kk }=bx_{n\kk }$ in the nilpotent
realization, so that we can ignore each matrix $M_\kk $ and $H_{i,\kk }$ in (\ref{JacB2}), since their
 coefficients vanish at the nilpotent realization.
Let $\mathbf{s}_\kk =[0,0,S_{3,\kk },S_{4,\kk },\ldots,S_{n+1,\kk }]^T$
and $\mathbf{t}_{i\kk }=[0,0,0,0,T_{5,i,\kk },T_{6,i,\kk },\ldots,T_{n+1,i,\kk }]^T.$
By (\ref{pB}), the column of $N$
corresponding to $y_\kk $ is
$$\overrightarrow{N}_{y_\kk }=\mathbf{s_\kk }+\sum_{i=1}^{\kk -1}x_{ni}\mathbf{t}_{i\kk }-\sum_{i=\kk +1}^nx_{ni}\mathbf{t}_{i\kk },$$ 
and the column corresponding to $x_{n\kk }$ is 
$$\overrightarrow{N}_{x_{n\kk }}=-b\mathbf{s_\kk }-\sum_{i=1}^{\kk -1}y_{i}\mathbf{t}_{i\kk }+
\sum_{i=\kk +1}^ny_{i}\mathbf{t}_{i\kk }$$ for $1\leq \kk \leq n-1$.
Thus, evaluated at the nilpotent realization with  $y_i= ba_{ni}=bx_{ni}$,
$\overrightarrow{N}_{x_{n\kk }}=-b\overrightarrow{N}_{y_\kk }$ for $1\leq \kk \leq n-1.$
Further $\overrightarrow{N}_{y_n}=\sum_{i=1}^{n-1}x_{ni}\mathbf{t}_{in}$ with
 $\overrightarrow{N}_{x_{nn}}=\sum_{i=1}^{n-1}(-y_{i})\mathbf{t}_{in}$. 
It follows that $N\vert_{X_B=B}$ is column equivalent 
to $[\ O\ |\overrightarrow{N}_{y_1}|\overrightarrow{N}_{y_2}|\cdots|\overrightarrow{N}_{y_{n}}].$

From (\ref{H1}), with $z=0$, the $(n+1)$ entry of $\overrightarrow{N}_{y_\kk }$ is  
$$(-1)^{\kk }\det X_A(1,\kk ),$$
which evaluated at $X_B=B$ is 
$$
(-1)^{\kk }\det A(1,\kk ).
$$
Thus, the hypothesis that 
there exists an index $v$ such that $a_{\elb v}\neq 0$ and $\det{A(j,v)}\neq 0$ 
implies
that the $(n+1)$ entry of $\overrightarrow{N}_{y_v}$ 
is nonzero. 
It follows that
$J_{X_B=B}$ is equivalent to
\begin{equation}\label{J2}
\left[\begin{array}{cc}
J_{X_A=A} & *\\
\mathbf{0}^T&\mathbf{r}^T
\end{array}\right]
\end{equation}
for some $\mathbf{r}\neq \mathbf{0}$. Hence $B$ allows a full-rank Jacobian.
Thus by Theorem~\ref{SAP}, every superpattern of $\cB$ is spectrally 
arbitrary. 
\end{proof}

 \begin{example}\label{Three}\rm{ 
Starting with $\T_2$, the unique spectrally arbitrary pattern of order $2$ up to equivalence \cite{DJOD}, 
the bordering technique of Theorem~\ref{borderT2} 
gives spectrally arbitrary patterns of order $3$.
In particular, consider the nilpotent matrix
\[A=\left[
\begin{array}{rr}
1&-1\\
1&-1\\
\end{array}\right]\in Q(\T_2)
\rm{\quad and\ let \quad}
X_A=\left[
\begin{array}{rr}
x_1&-1\\
x_2&-1\\
\end{array}\right].\]
Then $J_{X=A}$ has full rank and entry $a_{12}$ is non-Jacobian in $X_A$. 
Thus, the bordering technique of Theorem~\ref{borderT2} gives the matrix
\[\left[
\begin{array}{rcr}
1&-1-b&1\\
1&-1&0\\
b&-b&0
\end{array}\right],\]
providing different spectrally arbitrary patterns depending on the chosen
value of $b\neq 0$. Taking $b=\frac{1}{2}$ gives a sign pattern equivalent to $\W_3$ (see \cite{Britz}) 
with a full-rank Jacobian. 
Taking $b=-\frac{1}{2}$ gives a pattern equivalent to a superpattern of $\V_3$ (see \cite{Britz}) with
a full-rank Jacobian. Taking $b=-1$ gives a pattern equivalent to $\V_3$ with 
a full-rank Jacobian. This last option, using $b=a_{12}$ maintains sparsity (i.e, 
it gives a minimal spectrally arbitrary pattern.)
}\end{example}

\begin{remark}{\rm
With a well-chosen example, Theorem~\ref{borderT2} can be applied recursively. In particular, note that in (\ref{J2}),
the $n$ variables of $X_A$ are used to show that $B$ allows a full-rank Jacobian. 
Thus, at most one of the nonzero entries in row $n+1$ of $B$ needs to be Jacobian in $X_B.$
Further, an entry in row $n+1$ that is Jacobian in $X_B$ can be chosen to be any nonzero position $(n+1,v)$ for which 
$\det A(j,v)$ is nonzero. Note that, since the last column of $B$ has only one nonzero entry,
 $\det B(n+1,v)=(-1)^{j+n} \det A(j,v)$. Thus, if 
 there is more than one $v$ with $a_{kv}\neq 0$, and $\det A(j,v)\neq 0$, then bordering can be 
repeated recursively, applying it to $(j,k)=(n+1,k)$ in $B$.
}\end{remark}

\begin{example}\label{KN}{\rm
Consider the nilpotent realization 
$$\left[\begin{array}{rrr}
1&-1&0\\
1&0&-1 \\
1&0&-1
\end{array}\right]
$$
of the spectrally arbitrary pattern $\V_3$.
By applying Theorem~\ref{borderT2} with $b=k=1$, $j=3$ and $v=2$,  
and repeating recursively, increasing $j$ but keeping $b=k=1$ and $v=2$, 
a spectrally arbitrary pattern $\K_n$ is obtained for $n\geq 4$, with
\[\K_n=\left[\begin{array}{ccccccc} 
+&-&&&&&\\
+&0&-&&&O&\\
0&0&-&+&&&\\
0&-&0&0&+&&\\
\vdots&\vdots&\vdots&\vdots&\ddots&\ddots&\\
0&-&0&0&\cdots&0&+\\
+&-&0&0&\cdots&0&0
\end{array}
\right].\]
The nonzero entries in a nilpotent realization of $\K_n$ have magnitude $1$. 
}\end{example}

As far as we know, the spectrally arbitrary sign patterns $\B_n$ in Example~\ref{BN}
and $\K_n$ in Example~\ref{KN} have not previously appeared in the literature.

\medskip

\section{General bordering for $n=3$}

In Theorems~\ref{borderT} and \ref{borderT2}, we restricted to bordering 
with standard unit vectors in the place of $\mathbf{x}$ and $\mathbf{z}$ in
(\ref{B}). 
We next illustrate the more general bordering (\ref{B}) with a couple of
examples.

\begin{example}\label{T3}\rm{
Starting with the nilpotent realization $A$ of $\T_2$ given in
Example~\ref{Three}, a nilpotent realization of $\T_3$ can be obtained
as follows:
\begin{eqnarray*}
B=\left[
\begin{array}{rr|r}
1&0&0\\
0&1&0\\ \hline
-\frac{1}{2}&1&1\\
\end{array}\right]
\left[
\begin{array}{rr|r}
1&-1&0\\
1&-1&-1\\ \hline
0&0&0\\
\end{array}\right]
\left[\begin{array}{rr|r}
1&0&0\\
0&1&0\\ \hline
\frac{1}{2}&-1&1\\
\end{array}\right]&=&
\left[\begin{array}{rrr}
1&-1&0\\
\frac{1}{2}&0&-1\\
0&\frac{1}{2}&-1\\
\end{array}\right]\in Q(\T_3).
\end{eqnarray*}
Matrix $B$ allows a full-rank Jacobian and hence (as is well-known \cite{DJOD}) every superpattern
of $\T_3$ is spectrally arbitrary by Theorem~\ref{SAP}.
}\end{example}

\begin{example}\label{U3} \rm{Starting with the nilpotent realization $A$ of $\T_2$ given in
Example~\ref{Three}, a nilpotent realization of $\U_3$ (see \cite{Britz}) can be obtained as follows:  
\begin{eqnarray*}
B=\left[
\begin{array}{rr|r}
1&0&0\\
0&1&0\\ \hline
-1&2&1\\
\end{array}\right]
\left[
\begin{array}{rr|r}
1&-1&\frac{1}{2}\\
1&-1&0\\ \hline
0&0&0\\
\end{array}\right]
\left[\begin{array}{rr|r}
1&0&0\\
0&1&0\\ \hline
1&-2&1\\
\end{array}\right]&=&
\left[\begin{array}{rrr}
\frac{3}{2}&-2&\frac{1}{2}\\
1&-1&0\\
\frac{1}{2}&0&-\frac{1}{2}\\
\end{array}\right].
\end{eqnarray*}
In particular, matrix $B$ is a nilpotent realization of $\U_3$ that allows a full-rank Jacobian
and hence every superpattern
of $\U_3$ is spectrally arbitrary by Theorem~\ref{SAP}.
}\end{example}

As demonstrated in \cite{Britz}, every  spectrally arbitrary 
sign pattern of order $3$ is a superpattern of one of the four patterns 
$\T_3$, $\U_3$, $\V_3$ and $\W_3$.  
From Example~\ref{Three}, every superpattern 
of $\V_3$ and $\W_3$ is spectrally arbitrary by Theorem~\ref{borderT2} using a standard unit
bordering of $\T_2$.  The other two order $3$
patterns can be obtained by using a general bordering of $\T_2$ as demonstrated in 
Examples~\ref{T3} and \ref{U3}.

\begin{corollary}
Every spectrally arbitrary sign pattern of order $3$ is a superpattern of a pattern obtained from $\T_2$ 
by bordering as in $(\ref{B})$.
\end{corollary}

\section{Inertially arbitrary borderings}

We conclude by extending the main results in Sections~\ref{unitequal} and \ref{unitunequal} to
obtain inertially arbitrary sign patterns.
The \emph{inertia} of a matrix $A$ is the ordered triple $i(A)=(a,b,c)$ for which $a$ is the number of
eigenvalues of $A$ with positive real parts, $b$ is the number with negative real parts, and
$c$ is the number of eigenvalues with real parts zero. The \emph{refined inertia} of a matrix $A$ 
is the ordered $4$-tuple $ri(A)=(a,b,c_1,c_2)$ for which $c_1$ is the algebraic multiplicity of zero as an eigenvalue
for $A$ and $c_1+c_2=c$. Then $c_2$ is the number of nonzero imaginary eigenvalues of $A$. 
A sign pattern $\A$  of order $n$ is \emph{inertially arbitrary} if, for every non-negative integer choice of 
$(a,b,c)$ with $a+b+c=n$, there is some matrix $A\in Q(\A)$ with $i(A)=(a,b,c)$.
As with nilpotent matrices, a matrix $A$ of order $n$ with refined inertia 
$(0,0, c_1,c_2)$ \emph{allows a full-rank Jacobian} if the
Jacobian matrix $J_{X=A}$ has rank $n$.

The next theorem combines \cite[Theorem $2.13$]{CF} and \cite[Corollary $4.5$]{CGKOVV}.

\begin{theorem}\label{NJ-IAP}
Let $\A$ be a sign pattern and $A\in Q(\A)$ be a matrix with $\ri{A}=(0,0,c_1,c_2)$ for some $c_1\geq 2$.
If $A$ allows a full-rank Jacobian,
then 
\begin{enumerate}
\item[(a)] $A$ is nonderogatory, and
\item[(b)] every superpattern of $\cA$ is inertially arbitrary.
\end{enumerate}
\end{theorem}
Note that, unlike the context of Theorem~\ref{SAP}, $A$ is not
necessarily irreducible if $A$ allows a full-rank Jacobian in Theorem~\ref{NJ-IAP}.

\begin{example}\label{Ione}{\rm
 Let  
$$A=\left[\begin{array}{rrrr}
1&-1&0&0\\
1&-1&0&0\\
0&0&1&-2\\
0&0&1&-1\\
\end{array}
\right] \qquad {\rm and \qquad}
 X_A=\left[\begin{array}{rrrr}
1&x_1&0&0\\
1&x_2&0&0\\
0&0&1&x_3\\
0&0&1&x_4\\
\end{array}
\right]. 
 $$
This matrix $A \in Q(\T_2\oplus\T_2)$ is a nonderogatory  reducible matrix with refined inertia $(0,0,2,2)$ 
and $J_{X_A=A}$ has rank $4$. Therefore, by Theorem~\ref{NJ-IAP} every superpattern of
$$\A=\left[\begin{array}{rrrr}
+&-&0&0\\
+&-&0&0\\
0&0&+&-\\
0&0&+&-\\
\end{array}\right]$$
is inertially arbitrary. Note that while it is known  \cite{C} that
$\T_2 \oplus \T_2$ is spectrally arbitrary (and hence inertially arbitrary), it is not
yet known if every superpattern of $\T_2 \oplus \T_2$ is spectrally arbitrary.
}\end{example}

The proof of the next theorem is the same as that for Theorem~\ref{borderT}
except it uses Theorem~\ref{NJ-IAP} instead of Theorem~\ref{SAP}.

\begin{theorem}\label{borderIT}
Let $\A$ be a sign pattern. Suppose $A=[a_{ij}]\in Q(\A)$ is a matrix 
having refined inertia $(0,0,c_1,c_2)$  
with $c_1\geq 2$, and $A$ allows a full-rank Jacobian. 
Suppose $a_{kk}\neq0$ and $a_{kv}\neq 0$ for some $v\neq k$. 
If $\det{A(k,v)}\neq 0$, 
then $B$  in $(\ref{Be})$ has refined inertia $(0,0,c_1+1,c_2)$, 
$B$ allows a full-rank Jacobian and 
 every superpattern of $\cB=\sgn(B)$ is inertially arbitrary. 
\end{theorem}

\begin{example}{\rm
Let $A$ be the matrix in Example~\ref{Ione}.
 With $k=2$ and $v=1$, Theorem~\ref{borderIT}
implies that every superpattern of 
$$\B=\left[\begin{array}{rrrrr}
+&-&0&0&0\\
+&0&0&0&+\\
0&0&+&-&0\\
0&0&+&-&0\\
-&0&0&0&-\\
\end{array}\right]$$
is inertially arbitrary.  Note that $\B$ is spectrally arbitrary since
$\B$ is equivalent to $\T_2 \oplus \V_3$, but it is not known 
if every superpattern of $\B$ is spectrally arbitrary.
}\end{example}

The proof of the next theorem is the same as that for Theorem~\ref{borderT2}
except it uses Theorem~\ref{NJ-IAP} instead of Theorem~\ref{SAP}.

\begin{theorem}\label{borderIT2}
Let $\A$ is a sign pattern. Suppose $A=[a_{ij}]\in Q(\A)$ is a matrix with refined inertia $(0,0,c_1,c_2)$ 
for some $c_1\geq 2$ and $A$ allows a full-rank Jacobian.
Suppose $a_{j\elb }$, $j\neq \elb$, is non-Jacobian for some choice of $X_A$. 
If $a_{\elb v}\neq0$ 
and $\det{A(j,v)}\neq 0$, for some $v$, 
then $B$  in $(\ref{B2})$  has refined inertia $(0,0,c_1+1,c_2)$,
$B$ allows a full-rank Jacobian, and
every superpattern of $\cB=\sgn(B)$ is inertially arbitrary. 
\end{theorem}

\begin{example}{\rm Consider the matrix 
$$A=
\left[\begin{array}{rrrrr}
-1 &-1&-1&0&0\\
2&1&1&0&0\\
0&0&0&-1&-1\\
0&-1&0&0&-1\\
-1&0&0&0&0
\end{array}\right], {\rm{\quad and \quad}}
X_A=
\left[\begin{array}{rrrrr}
-1 &x_1&-1&0&0\\
2&x_2&x_3&0&0\\
0&0&0&-1&x_4\\
0&x_5&0&0&-1\\
-1&0&0&0&0
\end{array}\right].
$$
Matrix $A$ has sign pattern $\G_5$ from \cite{KOV} (see also Section 5.3 of \cite{CGKOVV}), $A$ has refined inertia $(0,0,3,2)$, 
$J_{X=A}$ has full rank and entry $(1,3)$ is non-Jacobian. Thus
with $j=1$, $k=3$, $v=4$, and $b=-1$ in Theorem~\ref{borderIT2}, we obtain the 
inertially arbitrary pattern
$$\B=\left[\begin{array}{rrrrrr}
- &-&0&0&0&+\\
+&+&+&0&0&0\\
0&0&0&-&-&0\\
0&-&0&0&-&0\\
-&0&0&0&0&0\\
0&0&0&+&+&0
\end{array}\right].$$
Since $\G_{5}$ has no nilpotent realization, it follows that $\B$ has no nilpotent
realization; thus $\B$ is not spectrally arbitrary.  
Note that, for $n\geq 2$, using the sign pattern $\G_{2n+1}$ with matrix $\tilde{A}_{2n+1}$ as 
listed in Section 5.3 of \cite{CGKOVV}, then $\tilde{A}_{2n+1}$ has
refined inertia $(0,0,2n-1,2)$, $\det(\tilde{A}(1,4))\neq 0$, and entry $(1,3)$ is non-Jacobian. Thus, using Theorem~\ref{borderIT2} with $j=1$, $k=3$, $v=4$, and $b=-1$ applied to $\tilde{A}_{2n+1},$ we
can construct an even order inertially arbitrary sign pattern with no nilpotent realization
for each even order $2n+2\geq 6$. In \cite{KOV}, only odd order sign patterns were provided with these conditions.}
\end{example}

\textbf{Acknowledgements.}
This research was initiated when the third author visited the University of 
Victoria with support from the Pacific Institute for Mathematical Sciences (PIMS).
The research is partially supported by the authors' NSERC Discovery Grants. 
The authors thank an anonymous referee 
for comments that helped to clarify parts of this paper.

\end{document}